\documentclass[11pt]{amsart}

\usepackage[margin=1in]{geometry}

\usepackage{xspace}
\usepackage{caption}
\usepackage{comment}
\usepackage{amssymb, amsmath, amsfonts, amsthm, graphics, mathrsfs, tikz-cd, array}
\usepackage{mathtools}
\usepackage{multirow}
\usepackage[shortlabels]{enumitem}
\usepackage[all]{xy}
\usepackage{hyperref}
\hypersetup{
colorlinks=true,
linkcolor=blue,
citecolor=magenta,
linktoc=page
}

\usetikzlibrary{decorations.pathreplacing,calligraphy}
\usetikzlibrary{hobby}
\usetikzlibrary{tqft}
\usepackage{subcaption}
\usepackage{indentfirst}
\usepackage{float}
\usepackage{faktor}

\renewcommand{\emptyset}{\varnothing}

\theoremstyle{plain}
\newtheorem{theorem}{Theorem}[section]
\newtheorem{theoremx}{Theorem}
\newtheorem{example}[theorem]{Example}

\newtheorem{proposition}[theorem]{Proposition}
\newtheorem{lemma}[theorem]{Lemma}
\newtheorem{corollary}[theorem]{Corollary}

\theoremstyle{definition}
\newtheorem{definition}[theorem]{Definition}

\newtheorem{remark}[theorem]{Remark}

\newcommand{\D}{{\mathbf{D}}}

\renewcommand{\H}{{\mathbf{H}}}

\renewcommand{\P}{{\mathbf{P}}}

\newcommand{\R}{{\mathbf{R}}}

\newcommand{\cC}{{\mathcal{C}}}

\newcommand{\cF}{{\mathcal{F}}}

\newcommand{\PGL}{\operatorname{PGL}}
\newcommand{\Hom}{\operatorname{Hom}}
\newcommand{\Int}{\operatorname{Int}}

\newcommand{\question}[1]{
\noindent
\framebox{\parbox{\textwidth}{{\textbf{Question:} #1}}}
}

\newcommand{\probref}[1]{\hyperref[prob:#1]{#1}}

\let\ol\overline

\title[Bijections on strictly convex sets and surfaces that preserve complete geodesics]{Bijections on strictly convex sets and closed convex projective surfaces that preserve complete geodesics}
\author{Drimik Roy Chowdhury}
\address{Department of Mathematics, University of Michigan Ann Arbor, MI 48109}
\email{drimikr@umich.edu}

\begin{document}

\begin{abstract}

In this paper, we study bijections on strictly convex sets of $\R\P^n$ for $n \geq 2$ and closed convex projective surfaces equipped with the Hilbert metric that map complete geodesics to complete geodesics as sets. Hyperbolic $n$-space with its standard metric is a special example of the spaces we consider, and it is known that these bijections in this context are precisely the isometries \cite{jeffers}. We first prove that this result generalizes to an arbitrary strictly convex set. For the surfaces setting, we prove the equivalence of mapping simple closed geodesics to simple closed geodesics and mapping closed geodesics to closed geodesics. We also outline some future directions and questions to further explore these topics.
\end{abstract}

\maketitle
\vspace{-2em}
\tableofcontents
\vspace{-3em}
\section{Introduction}
We want to study the behavior of bijections on certain metric spaces that preserve complete geodesics as sets. The general motivation for this investigation began with the results described in \cite{jeffers}, which asked if more could be said about bijective maps on each of $\R^n, \H^n$ and $\R^n \cup \{\infty\}$ if they preserved complete geodesics with respect to their standard metrics. On its own merit, the question to what extent complete geodesics characterize your space (or perhaps, more specifically, what attributes of your space may be determined) stands as an intriguing premise. For example, \cite{jeffers} proves that such a map on $\R^n$ must be affine and on $\H^n$ must be an isometry. In this sense, we see that preserving geodesics as sets imposes very rigid constraints on the function in study. 

The notion of a convex set is ubiquitous in mathematics; they are used in areas such as number theory, geometry, dynamical systems and optimization. Hilbert identified a metric one could construct on a properly convex set in $\R\P^n$ based on the cross ratio where the resulting spaces were metrically and geodesically complete. If the convex set is strictly convex, then geodesics are projective lines of $\R\P^n$ restricted to our set. In the special case for when the convex set is a ball, or more generally an ellipsoid, the Hilbert metric coincides with the Klein-Beltrami model of hyperbolic space \cite[Introduction]{handbook_hilbert_geometry}. Hence, we see that Hilbert geometry includes hyperbolic geometry. Moreover, since the cross ratio is invariant under projective transformations, we see that the Hilbert metric is a canonical attribute of a properly convex set of $\R\P^n$ i.e., independent of the affine chart in study. Hilbert geometry provides a rich mechanism to study questions from convexity theory and has deep connections to differential geometry via Finsler geometry.

Our goal is to characterize properties of bijections on strictly convex sets in $\R\P^n$ and closed convex projective surfaces with the Hilbert metric that preserve complete geodesics as sets.

We now state our main result in the strictly convex set case, which is that a bijection preserving complete geodesics must be an isometry. Thus, our paper directly generalizes the known result in the hyperbolic scenario mentioned above. The following fact is Theorem \ref{thm:main_omega} in the paper. 
\begin{theoremx}
For $n \geq 2$, let $\Omega \subset \R\P^n$ be a strictly convex set endowed with the Hilbert metric and let $f\colon \Omega \to \Omega$ be a bijection, mapping complete geodesics to complete geodesics as sets. Then $f$ is an isometry of $\Omega$ and can be extended to a projective transformation of $\R\P^n$ via an isometric embedding of $\Omega$ into $\R\P^n$.
\end{theoremx}
We prove this by showing any $f\colon \Omega \to \Omega$ as in the theorem must preserve the cyclic order of points along any complete geodesic and hence, must preserve geodesic segments. Theorem 2.17 of \cite{avidan} then tells us that $f$ must be a fractional linear map, and this essentially concludes the argument. See Section \ref{sec:prelims_convex} for the precise definitions of the aforementioned terms and relevant background. Section \ref{sec:main_convex} states and proves this main result.

Following this, we visit closed (real) convex projective surfaces. Let $S$ be such a surface and $f\colon S \to S$ be a bijection mapping complete geodesics to complete geodesics as sets. Our conjecture was that $f$ must be an isometry. To explore this, we first asked what would be the consequences of this result. In particular, $f$ would be continuous and so it would map closed geodesics to closed geodesics. Our main result in the surfaces setting is that preserving closed geodesics is equivalent to preserving simple closed geodesics. The exact statement, Theorem \ref{thm:main}, is quite technical, and so we have chosen to omit it here. Section \ref{sec:prelims_surface} discusses the necessary background for our exploration on convex projective surfaces and Section \ref{sec:main} states the main result. Whether or not $f\colon S \to S$ is an isometry is still an open question.

Section \ref{sec:geodesic_laminations} discusses geodesic laminations and a context in which the main result of our exploration on surfaces holds. Lastly, Section \ref{sec:future_directions} outlines notable questions to follow from the work conducted in this paper.

\subsection*{Acknowledgements}
The author would like to thank his mentors Giusppe Martone and Ralf Spatzier for their constant support in this exploration. Moreover, the author appreciates the University of Michigan Math department for making this project possible.


\section{Preliminaries on strictly convex sets}\label{sec:prelims_convex}
First, for some standard definitions.
\begin{definition}[Convex sets]
For $n \geq 2$, we call a domain $\Omega \subset \R\P^n$ contained in an affine patch \emph{convex} if the intersection of $\Omega$ with every affine line in $\R^n$ is connected. In addition, we say a convex domain $\Omega$ is 
\begin{itemize}
\item \emph{properly convex}, if the closure $\ol{\Omega}$ is convex and contained within the complement $\R^n = \R\P^n - \R\P^{n-1}$ of some $\R\P^{n-1}$ linearly embedded in $\R\P^n$;
\item \emph{strictly convex}, if $\Omega$ is properly convex and the boundary $\partial \Omega$ contains no nontrivial line segments.
\end{itemize}
\end{definition}

\noindent{\bf Note.} Any convex set for us will be strictly convex.

\begin{definition}[Hilbert metric]
Let $\Omega$ be a properly convex set in $\R\P^n$. The \emph{Hilbert metric of $\Omega$}, $d_\Omega$, is defined as follows: for any $x,y \in \Omega$, the Hilbert distance is
$$
d_\Omega(x,y) = \log\left(\frac{|ay|}{|ax|} \frac{|bx|}{|by|}\right)
$$
where $a,b \in \partial \Omega$ are endpoints of any line passing through $x$ and $y$ and $|ay|$ denotes the usual Euclidean distance between $a$ and $y$. Then $(\Omega,d_\Omega)$ defines a \emph{Hilbert metric space}. The order of points considered is $a,x,y,b$ as seen in Figure \ref{fig:hilbert_metric}. The argument to logarithm above is known as the \emph{cross ratio} of the points $a,x,y,b$. 
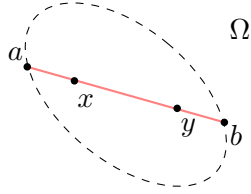
\begin{figure}[H]
\centering
\begin{tikzpicture}[scale=0.5,rotate=30]
\draw[rotate=20,dashed] (0,0) ellipse (50pt and 90pt); 
\draw[line width=0.8pt,draw=red!50] (-1.9,1.95) -- (1.9,-1.95);
\node at (-1,1)[circle,fill,inner sep=1pt]{};
\node at (1,-1)[circle,fill,inner sep=1pt]{};
\node at (1.9,-1.95)[circle,fill,inner sep=1pt]{};
\node at (-1.9,1.95)[circle,fill,inner sep=1pt]{};
\node at (3.5,0) {$\Omega$};
\node at (1,-1.6) {$y$};
\node at (-1,0.4) {$x$};
\node at (2,-2.4) {$b$};
\node at (-2,2.4) {$a$};
\end{tikzpicture}
\caption{Computing the Hilbert Metric}
\label{fig:hilbert_metric}
\end{figure}
\end{definition}
We note that if $x \neq y$ are in $\Omega$, then there exists a unique line between $x,y$ and if $x = y$, then the $d_\Omega(x,x) = 0$ with regards to any line through $x$. In other words, the Hilbert metric on $\Omega$ is a well-defined notion. Moreover, the definition of the metric structure on $\Omega$ is independent of the choice of the affine chart in $\R\P^n$ in which $\Omega$ is contained since the charts are related by a projective transformations and these maps preserve the Hilbert metric. Any bounded convex set in $\R^n$ can be given the notion of a Hilbert distance but it may not be a metric space (for example, if the set contains points of its boundary); however, the relative interior is a complete metric space for this distance \cite[Introduction]{handbook_hilbert_geometry}. Hence, we make a simple assumption to only consider open convex sets. Projective lines are always geodesics for a general Hilbert geometry \cite[\S 2.3.2]{crampon2014geodesic}, but the strictly convex assumption tells us that $(\Omega,d_\Omega)$ is uniquely geodesic \cite[Corollary 12.7]{papadopoulos2014funk}; thus, the only geodesics for the strictly convex setting are conventional straight lines of $\R^n$ restricted to the set (viewing $\Omega$ in an affine chart). Lastly, the topology induced by the Hilbert metric coincides with the subspace Euclidean topology on $\Omega$ \cite[Proposition 6.1; \S 12]{papadopoulos2014funk}.


Examples of maps that preserve complete geodesics would include isometries of the Hilbert metric. Note that elements of $\PGL(n+1,\R)$ that fix $\Omega$ as a set are isometries because the cross ratio is a projective invariant. It turns out that in the strictly convex setting, these are all the isometries of $(\Omega,d_\Omega)$ \cite[Corollary 10.4]{marquis2013around}. An example of this is $\Omega = \D^n$ the hyperbolic $n$-disk contained in $\R\P^n$ and projective isometries being hyperbolic isometries.

Our main question is whether any bijection preserving complete geodesics must be an isometry:

\question{For $n \geq 2$, let $\Omega \subset \R\P^n$ be a strictly convex set endowed with the Hilbert metric and $f\colon \Omega \to \Omega$ be a bijection, mapping complete geodesics to complete geodesics as sets. Is $f$ an isometry for the Hilbert metric?}

Theorem \ref{thm:main_omega} shows that the answer to this question is in the affirmative. Moreover, this Theorem also proves that our map in study can be viewed as a projective transformation (upto an isometric embedding of $\Omega$ into $\R\P^n$).

\begin{definition}[Interval preserving]
Let $x,y \in \R^n$. A map $f\colon D \to \R^n$ where $D \subset \R^n$ is said to be \emph{interval preserving} if for all $[x,y] \subset D$, $f([x,y]) = [z,w]$ for some $z,w \in \R^n$ where $[x,y]$ denotes the line segment between $x,y$ including endpoints.
\end{definition}

\begin{definition}[Fractional linear maps]
Let $A: \R^n \to \R^n$ be linear, $b,c \in \R^n$ and $d \in \R$. For every such prescription, we get a map $\R^n \supset D \to \R^n$ as
$$
v \mapsto \frac{1}{\langle c, v \rangle + d} \left(Av + b\right)
$$
where $D$ is one of the half-spaces determined by the complement of the set $\{v: \langle c,v \rangle + d = 0\}$. In line with reference \cite{avidan}, we call such maps \emph{fractional linear maps} and the half-space $D$ is taken to be $\{v: \langle c,v \rangle > -d \}$.
\end{definition}

The following is a known result (also proven in \cite{shiffman}), and serves as a crucial tool to prove Theorem \ref{thm:main_omega}.
\begin{theorem}[\cite{avidan}, Theorem 2.17]\label{thm:interval_preserving_to_fractional}
Let $n \geq 2$. Let $K \subset \R^n$ be a convex set with non-empty interior. Suppose $f\colon K \to \R^n$ is an injective, interval-preserving map. Then $f$ is a fractional linear map.
\end{theorem}
The proof in \cite{avidan} works by showing the behavior of $f$ can be easily understood on simplices and on their intersections, and hence, on a covering of $K$ with these objects. The proof in \cite{shiffman} presents the same fact though by completely working in the projective case, by extending a map on a convex set to all of $\R\P^n$ using Desargues' theorem and then applying the fundamental theorem of projective geometry.

\begin{remark}\label{rmk:fractional_linear_invertible_equivalence}
\cite{avidan} states that a fractional linear map is injective if and only if the real matrix
$
\begin{pmatrix}
A & b\\
{}^t c & d
\end{pmatrix}
$
is invertible. It is not necessary for $A$ to be invertible, e.g., consider $\begin{pmatrix} 1 & 0 & 0\\ 0 & 0 & 1\\ 0 & 1 & 0 \end{pmatrix}$.
\end{remark}

\begin{remark}\label{rmk:replace_with_cnn}
As noted in Remark 2.20 of \cite{avidan}, the set $K$ in Theorem \ref{thm:interval_preserving_to_fractional} could actually be replaced with any open, connected subset of $\R^n$.
\end{remark}


\section{Main results for bijections on strictly convex sets that preserve complete geodesics}\label{sec:main_convex}
\begin{subsection}{Notation}\label{notation_convex}
We may refer to $\Omega$ as a bounded, open, strictly convex set in $\R^n$ without any harm. Moreover, we may use the terms lines and geodesics interchangeably since geodesics with respect to the Hilbert metric are the usual restriction of lines in $\R^n$. The following serves as important notation for this section:
\begin{table}[H]
\centering
\begin{tabular}{rcl}
$[x,y]$ (resp. $(x,y)$) & $:$ & line segment between $x,y$ including (resp. excluding) endpoints \\
$\Omega$ & $:$ & strictly convex set of $\R\P^n$ for $n \geq 2$\\
$f$ & $:$ & bijection on $\Omega$, mapping complete geodesics to complete geodesics as sets
\end{tabular}
\end{table}
\end{subsection}
Our main conclusion for the convex set case is the following theorem.
\begin{theorem}\label{thm:main_omega}
For $n \geq 2$, let $\Omega \subset \R\P^n$ be a strictly convex set endowed with the Hilbert metric and let $f\colon \Omega \to \Omega$ be a bijection, mapping complete geodesics to complete geodesics as sets. Then $f$ is an isometry of $\Omega$ and can be extended to a projective transformation of $\R\P^n$ via an isometric embedding of $\Omega$ into $\R\P^n$.
\end{theorem}
A particular corollary of this fact is that $f$ is continuous where defined. Theorem \ref{thm:main_omega} directly generalizes the result of \cite{jeffers} to any strictly convex set of $\R\P^n$, not just hyperbolic $n$-space $\H^n$. 

As a breakdown for our argument, we shall first show our function is interval preserving, which would tell us that we have a fractional linear map by Theorem \ref{thm:interval_preserving_to_fractional}. Section \ref{subsec:order_preserving} proves the necessary lemmas to show our function is \emph{order preserving along any geodesic} i.e., the cyclic order of any 3 points is maintained by our function. Section \ref{subsec:interval_preserving_continuity} utilizes the fact that our function is order preserving to conclude it is interval preserving, and we also provide an independent proof of continuity (without appealing to Theorem \ref{thm:interval_preserving_to_fractional}). We note that an argument of interval preserving implying continuity is proven in \cite{avidan}; our argument runs in mostly the same fashion. We simply include it here for the sake of cohesion. Lastly, Section \ref{subsec:interval_preserving_to_isometry} concludes our proof of Theorem \ref{thm:main_omega} using the properties and lemmas established. 

\subsection{Proving order preserving along any geodesic}\hrulefill
\label{subsec:order_preserving}

Before we proceed, we note quickly that for any bijection preserving complete geodesics on $\Omega$, its set-inverse also has the same properties. This is the content of Lemma \ref{lem:inverse_same_convex}, which really is just a consequence of the uniqueness of (complete) geodesics between any two points, and this in turn is because we are dealing with only strictly convex sets.

\begin{lemma}\label{lem:inverse_same_convex}
For $n \geq 2$, let $\Omega \subset \R\P^n$ be a strictly convex set endowed with the Hilbert metric and $f\colon \Omega \to \Omega$ be a bijection, mapping complete geodesics to complete geodesics as sets. Then $f^{-1}$ sends complete geodesics to complete geodesics as sets.
\end{lemma}
\begin{proof}
Let $\ell' \subset \Omega$ be a line and $y,y' \in \ell'$ be distinct points. There exists distinct unique $x,x' \in \Omega$ such that $f(x) = y, f(x') = y'$. Since any two distinct points in $\Omega$ determine a unique line, let $\ell$ be the line passing through $x,x'$. Hence, $f(\ell)$ by assumption is a line through $y,y'$ and $\Omega$ being uniquely geodesic tells us that it must be the case that $f(\ell) = \ell'$ so $f^{-1}(\ell') = \ell$.
\end{proof}

Lemma \ref{lem:separate_line_existence} constructs objects in $\Omega$ that essentially disconnect our space. As we shall see, these entities create fundamental obstacles that force our bijections in study to be order preserving.

\begin{definition}[Subspaces of strictly convex sets]
For $n \geq 2$, let $\Omega \subset \R\P^n$ be a strictly convex set endowed with the Hilbert metric. Then a \emph{$k$-dimensional subspace of $\Omega$} is a nonempty intersection of an affine $k$-dimensional subspace of $\R^n$ with $\Omega$. We may call a \emph{plane} of $\Omega$ as a $2$-dimensional subspace of $\Omega$ and a \emph{hyperplane} as an $(n-1)$-dimensional subspace of $\Omega$. 
\end{definition}

\begin{lemma}\label{lem:separate_line_existence}
For $n \geq 2$, let $\Omega \subset \R\P^n$ be a strictly convex set endowed with the Hilbert metric. Let $a,b,c$ be distinct collinear points along a line $L \subset \Omega$ with the cyclic order of points as $(a,b,c)$. Then there exists a line $L_a$ and hyperplanes $H_b,H_c$ of $\Omega$ with the following properties: 
\begin{enumerate}[(a)]
\item $a \in L_a, b \in H_b, c \in H_c$
\item $H_b$ intersects $H_c$ 
\item $L_a$ does not intersect $H_b$ or $H_c$
\end{enumerate}
Note that $(c)$ forces $L_a \neq L$ and $L \nsubseteq H_b$, $L \nsubseteq H_c$.
\end{lemma}
\begin{proof}
We first prove the case for $n = 2$. Hence, $H_b, H_c$ are just lines and we shall denote them $L_b, L_c$, respectively. Figure \ref{fig:La_Hb_Hc} provides a picture for the forthcoming discussion. Let $L \subset \Omega$ be a line and points $a,b,c$ as outlined in the assumption. Pick a point $x \in \partial \Omega$ on one side of $\Omega$ determined by $L$ (so $x \not\in \partial L \cap \partial \Omega$) and draw $L_a$ as the line between $a$ and $x$. Similarly, we have $L_c$ between $c$ and $x$, and let $z$ denote the other endpoint of this line on $\partial \Omega$. Then $L_a$ hits $\partial \Omega$ on the opposite side of $L$ at a point $y$. Let $L_b$ be the line between $b$ and $y$. Note that all these lines live inside $\Omega$ by convexity. We claim that these lines satisfy our criteria.
\begin{figure}[H]
\centering
\begin{tikzpicture}[scale=0.6]
\draw[rotate=20,dashed] (0,0) ellipse (130pt and 90pt); 
\draw[line width=0.8pt,draw=blue!50] (-3,-3) -- (3,3);
\draw[line width=0.8pt,draw=red!50] (-2,-3.3) -- (-1,3);
\draw[line width=0.8pt,draw=red!50] (-2,-3.3) -- (2,3.3);
\draw[line width=0.8pt,draw=red!50] (4.4,1) -- (-1,3);
\node at (-2, -1.5) {$a$};
\node at (-1.75,-1.8)[circle,fill,inner sep=1.5pt] {};
\node at (0.5, 0) {$b$};
\node at (0,0)[circle,fill,inner sep=1.5pt] {};
\node at (2,1.5) {$c$};
\node at (1.95,1.9)[circle,fill,inner sep=1.5pt] {};
\node at (-1,3.4) {$x$};
\node at (-2,-3.8) {$y$};
\node at (4.7,1.2) {$z$};
\node[text=red] at (4,0.8) {$L_c$};
\node[text=red] at (-1,-2.8) {$L_b$};
\node[text=red] at (-1.5,2.5) {$L_a$};
\node[text=blue] at (3,2.5) {$L$};
\node at (4,3.3) {$\Omega$};
\end{tikzpicture}
\caption{Construction of lines $L_a,L_b,L_c \subset \Omega$ such that $L_a \cap L_b = L_a \cap L_c = \emptyset$ and $L_b \cap L_c \neq \emptyset$}
\label{fig:La_Hb_Hc}
\end{figure}
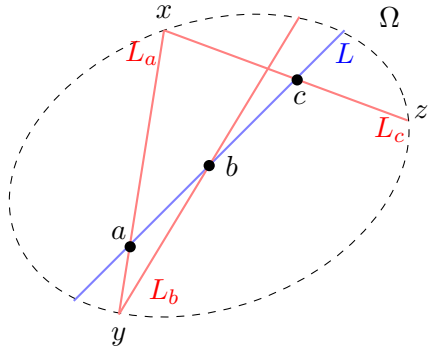
It is clear that $L_a$ does not intersect $L_b$ or $L_c$ since $\Omega$ is open. Since $b$ lies in the interval $(a,c)$, $L_b$ has its other endpoint on the component of $\ol{\Omega} \setminus \ol{L}$ not containing $y$. Since $\Omega$ is convex, we must hence have $L_b$ intersect $L_c$.
 For the general case $n > 2$, embed $L$ into a plane $P$ of $\Omega$ i.e., intersect $\Omega$ with a plane of $\R^2$ containing $L$ and call the intersection $P$. Then $P$ is strictly convex as it is the intersection of two convex sets and the boundary of $P$ is a subset of the boundary of $\Omega$. Hence, our case above gives us lines $L_a,L_b,L_c \subset P \subset \Omega$. The complement of the plane $P$ at $b$ gives us $n-2$ linearly independent directions. Hence, there exists hyperplanes $H_b,H_c$ of $\Omega$ such that $H_b \cap P = L_b$ and $H_c \cap P = L_c$. These hyperplanes with $L_a$ conclude our proof.
\end{proof}


\begin{lemma}\label{lem:preserve_subspace}
For $n \geq 2$, let $\Omega \subset \R\P^n$ be a strictly convex set endowed with the Hilbert metric and $f\colon \Omega \to \Omega$ be a bijection, mapping complete geodesics to complete geodesics as sets. Then $f$ takes $r$-dimensional subspaces to $r$-dimensional subspaces for $1 \leq r < n$. 
\end{lemma}
\begin{proof}
We prove the claim by induction. The $r=1$ case are exactly the lines of $\Omega$ and by assumption $f$ takes lines to lines. Assume the case for $r$-dimensional subspaces. Let $L_1,\ldots,L_{r+1}$ be lines in $\Omega$ that are linearly independent at $p \in \Omega$ i.e., the lines determine an $r+1$-dimensional subspace $H$ intersecting at $p \in \Omega$. We first show that $f(L_1),\ldots,f(L_{r+1})$ are linearly independent at $f(p)$ in $\Omega$. If this was not the case, then for some $i$, we would have that $f(L_i)$ belongs to the span of $\{f(L_j)\}_{j \neq i}$. But this latter space is $r$-dimensional and so by the inductive hypothesis applied to $f^{-1}$ (a bijection taking lines to lines), we must have that $L_i$ belong to the span of $\{L_j\}_{j \neq i}$ -- contradiction. This tells us that $f(H)$ contains a subspace of dimension at least $r+1$. Assume first that $f(H)$ is this subspace and so $\dim f(H) \geq \dim H = r+1$. 

If $\dim f(H) > \dim H$ i.e., $f(H)$ contains at least $r+2$ linearly independent directions, then we could find a line $\ell \subset f(H)$ passing through $f(p)$ and linearly independent to $f(L_1),\ldots,f(L_{r+1})$; this is because with at least $r+2$ linearly independent directions, there exists some line $P$ not contained in the span of $f(L_1),\ldots,f(L_{r+1})$. Hence, $\ell$ can be a line between $f(p)$ and a point on $P$ not in the span of $f(L_1),\ldots,f(L_{r+1})$. With this note, it suffices to prove that any line passing through $f(p)$ must be in the span of $f(L_1),\ldots,f(L_{r+1})$. Let $K'$ be any line passing through $f(p)$ in $f(H)$. Applying $f^{-1}$ to $K'$, we get a line $K \subset H$ that passes through $p$ and hence belongs to the span of $L_1,\ldots,L_{r+1}$. Since any line is determined by two points, it is enough to find two points of $K$ that are contained in the span of $L_1,\ldots,L_{r+1}$. We already have $p$. Let $q \in K$, $q \neq p$ and $r \in L_1$, $r \neq p$. Then convexity forces the (unique) line passing through $q,r$ to be contained in $H$; this provides the second point. Hence, $f(K) = K'$ lies in the span of $f(L_1),\ldots,f(L_{r+1})$.


We now complete the proof without the assumption that $f(H)$ is a subspace. We established that $f(H)$ contains a subspace of dimension at least $r+1$. Let $H' \subset f(H)$ denote this subspace. Then the above argument runs through to conclude $\dim H' = \dim H = r+1$. Now applying $f^{-1}$, we get $f^{-1}(H') \subset H$. Since $f^{-1}$ is a map of the same properties as $f$, we can apply our deductions from above again to conclude $f^{-1}(H')$ contains a subspace of dimension at least $r+1$. This collapses our inequalities to conclude $f^{-1}(H') = H \Rightarrow f(H) = H'$ as the only $r+1$-dimensional subspace of $H$ is $H$ itself.
\end{proof}

\begin{lemma}\label{lem:preserve_order}
For $n \geq 2$, let $\Omega \subset \R\P^n$ be a strictly convex set endowed with the Hilbert metric and $f\colon \Omega \to \Omega$ be a bijection, mapping complete geodesics to complete geodesics as sets. Then $f$ preserves the order of points along any line in $\Omega$.
\end{lemma}
\begin{proof}
First, we prove the case for $\Omega \subset \R\P^2$ being strictly convex. We proceed by contradiction. Let $L$ denote a line with distinct points $a,b,c$ with order $(a,b,c)$ such that $f$ interchanges the order of $a,b$ e.g., as depicted in Figure \ref{fig:preserve_order}. Create lines $L_a, L_b, L_c$ that pass through $a,b,c$, respectively, with the prescribed properties in Lemma \ref{lem:separate_line_existence}. By assumption $f(a)$ lies in $(f(b),f(c))$ and so $f(L_a)$ must pass through the interval $(f(b), f(c))$. Thus, the image of $L_a$ must intersect either $f(L_b)$ or $f(L_c)$ but this is a contradiction to the injectivity of $f$. 


\begin{figure}[H]
\centering
\begin{tikzpicture}
\begin{scope}[scale=0.6]
\draw[rotate=20,dashed] (0,0) ellipse (130pt and 90pt); 
\draw[line width=0.8pt,draw=blue!50] (-3,-3) -- (3,3);
\draw[line width=0.8pt,draw=red!50] (-2,-3.3) -- (-1,3);
\draw[line width=0.8pt,draw=red!50] (-2,-3.3) -- (2,3.3);
\draw[line width=0.8pt,draw=red!50] (4.4,1) -- (-1,3);
\node at (-2, -1.5) {$a$};
\node at (-1.75,-1.8)[circle,fill,inner sep=1.5pt] {};
\node at (0.5, 0) {$b$};
\node at (0,0)[circle,fill,inner sep=1.5pt] {};
\node at (2,1.5) {$c$};
\node at (1.95,1.9)[circle,fill,inner sep=1.5pt] {};
\node[text=red] at (4,0.8) {$L_c$};
\node[text=red] at (-1,-2.8) {$L_b$};
\node[text=red] at (-1.5,2.5) {$L_a$};
\node[text=blue] at (3,2.5) {$L$};
\node at (4,3.3) {$\Omega$};
\end{scope}
\begin{scope}[xshift=2cm,yshift=-1.5cm]
\draw[line width=0.8pt, ->] (1,1.8) -- (2.5,1.8);
\node at (1.75, 2.2) {$f$};
\end{scope}
\begin{scope}[scale=0.6,xshift=12.25cm]
\draw[rotate=20,dashed] (0,0) ellipse (130pt and 90pt); 
\draw[line width=0.8pt,draw=blue!50] (-4,0.8) -- (4,-0.8);
\draw[line width=0.8pt,draw=red!50] (-2,-3.3) -- (-1,3);
\draw[line width=0.8pt,draw=red!50] (-3,-3) -- (4.25,1.75);
\node at (0.5, 1.4) {$f(a)$};
\path[rotate=180,draw, decoration = {calligraphic brace, mirror}, decorate] (-1.325,0.1) -- (1.25,-0.4);
\node[rotate=-110] at (0.4,0.75) {$\in$};
\node at (-2.5, -0.25) {$f(b)$};
\node at (-1.4,0.3)[circle,fill,inner sep=1.5pt] {};
\node at (1.3,-1) {$f(c)$};
\node at (1.25,-0.25)[circle,fill,inner sep=1.5pt] {};
\node[text=blue] at (-2.6,1.25) {$f(L)$};
\node[text=red] at (0,2.75) {$f(L_b)$};
\node[text=red] at (3,2) {$f(L_c)$};
\node at (4,3.3) {$\Omega$};
\end{scope}
\end{tikzpicture}
\caption{Bijectivity of $f$ and Lemma \ref{lem:separate_line_existence} guarantee that $f$ is order preserving on $\Omega$}
\label{fig:preserve_order}
\end{figure}
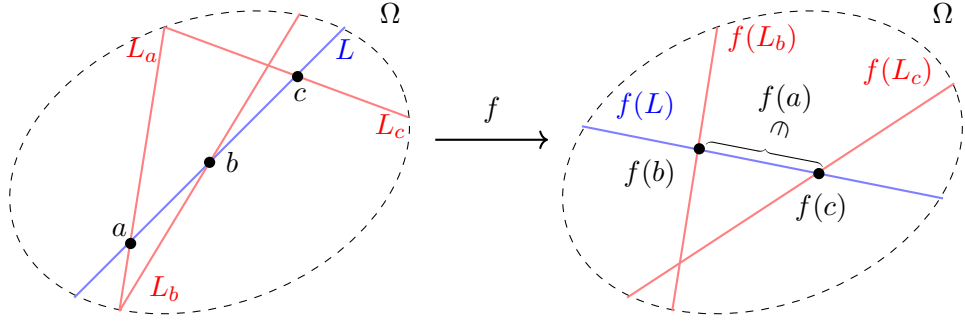
For the general case $n \geq 2$, via Lemma \ref{lem:separate_line_existence}, we similarly get a line $L_a \subset \Omega$ through $a$ and hyperplanes $H_b,H_c \subset \Omega$ containing $b,c$ such that $H_b \cap H_c \neq \emptyset$ and $L_a \cap H_b = L_a \cap H_c = \emptyset$. By Lemma \ref{lem:preserve_subspace}, we know that $f(L_a)$ is a line through $f(a)$ and $f(H_b), f(H_c)$ are hyperplanes containing $f(b), f(c)$, respectively. For emphasis, note that $f$ takes hyperplanes -- the subspaces that disconnect $\Omega$ -- to hyperplanes. Moreover, the intersection properties of $L_a,H_b,H_c$ hold under the image of $f$ since $f$ is a bijection. If $f(a)$ lies in the interval $(f(b), f(c))$, then the fact that $f(H_b)$ and $f(H_c)$ intersect forces $f(L_a)$ to intersect one of these two hyperplanes by convexity, which would again provide a contradiction. 
\end{proof}





\subsection{Proving interval preserving and continuity}\hrulefill
\label{subsec:interval_preserving_continuity}

We mainly wrap up a few results for our map in study with established facts from Section \ref{subsec:order_preserving}.
\begin{theorem}\label{thm:interval_preserving}
For $n \geq 2$, let $\Omega \subset \R\P^n$ be a strictly convex set endowed with the Hilbert metric and $f\colon \Omega \to \Omega$ be a bijection, mapping complete geodesics to complete geodesics as sets. Then $f$ is interval preserving and is a fractional linear map.
\end{theorem}
\begin{proof}
Since $f$ is injective, it suffices to show $f$ is interval preserving by Theorem \ref{thm:interval_preserving_to_fractional}. Let $x,y \in \Omega$. Then because $f$ preserves order of points by Lemma \ref{lem:preserve_order}, we have that $f([x,y]) \subset [f(x), f(y)]$. Now, note that $f^{-1}$ is a bijection, mapping lines to lines by Lemma \ref{lem:inverse_same_convex}, so that gives us the reverse inclusion.
\end{proof}

We shall now prove that any bijection on $\Omega$ preserving lines must be continuous. This could just be stated as a corollary of Theorem \ref{thm:interval_preserving_to_fractional} and Theorem \ref{thm:interval_preserving}, but we provide an independent proof of this fact using the machinery we have developed. This is the content of Theorem \ref{thm:continuity}.

 
\begin{theorem}\label{thm:continuity}
For $n \geq 2$, let $\Omega \subset \R\P^n$ be a strictly convex set endowed with the Hilbert metric and $f\colon \Omega \to \Omega$ be a bijection, mapping complete geodesics to complete geodesics as sets. Then $f$ is continuous.
\end{theorem}
\begin{proof}
Let $y \in \Omega$ and $B$ be an open neighborhood around $y$. Then there exists a unique $x \in \Omega$ such that $f(x) = y$. Necessarily, $x$ lies in the preimage of $B$. What we need to show is that there exists some open neighborhood $U$ of $x$ such that $U \subset f^{-1}(B)$ or equivalently, that $x$ belongs to the interior of $f^{-1}(B)$. We prove that either there exists some $x' \in \Omega$ distinct from $x$ such that $[x,x'] \cap f^{-1}(B) = \{x\}$ or $x \in \Int f^{-1}(B)$. If no such $x'$ exists then there exists a $\varepsilon > 0$ such that $B_\varepsilon(x) \setminus \{x\} \subset f^{-1}(B)$. Convexity now forces $x \in f^{-1}(B)$ and in particular, $x \in \Int f^{-1}(B)$ because $x \in B_{\varepsilon}(x) \subset \Int f^{-1}(B)$. Otherwise, we have an $x'$ with the properties noted. Applying Theorem \ref{thm:interval_preserving}, we see that $[f(x),f(x')] \cap B = \{f(x)\}$. But $f(x) = y$ and $B$ was chosen to be a open neighborhood around $y$ so any interval (with more than 1 point) containing $y$ must necessarily hit another point of $B$. Contradiction.
\end{proof}

\subsection{Recovering the disk result}\hrulefill\label{subsec:main_disk}

\begin{example}\label{thm:main_disk}
Let $\D^n$ denote the open unit $n$-disk in $\R^n$. Let $n \geq 2$ and $f\colon \D^n \to \D^n$ be a bijection, mapping complete geodesics to complete geodesics as sets. Then $f$ is a hyperbolic isometry.
\end{example}
We note that this is implied by Theorem \ref{thm:main_omega}. The following proof is mainly an involved example on applying properties of the hyperbolic disk with the fact that $f$ is a fractional linear map from Theorem \ref{thm:interval_preserving} without appealing to projective language. 
\begin{proof}
Let $f\colon \D^n \to \D^n$ be a bijection, mapping complete geodesics to complete geodesics as sets. By precomposing $f$ with a hyperbolic isometry, we can assume that $f(0) = 0$. By Theorem \ref{thm:interval_preserving}, $f$ is a fractional linear map. Hence, for any $x \in \D^n$
$$
f(x) = \frac{Ax + b}{ \langle c,x \rangle + d}
$$
for some fixed $A \in \Hom(\R^n,\R^n)$, $b,c \in \R^n$, and $d \in \R$. Since $f(0) = 0$, this forces $b = 0$. By Remark \ref{rmk:fractional_linear_invertible_equivalence}, we have that $A$ is an invertible linear map and $d$ is nonzero. Now, we claim that $\langle c,x \rangle + d \neq 0$ on $\partial \D^n$. Suppose for contradiction that such an $x \in \partial \D^n$ existed for which $\langle c,x \rangle + d = 0$. If $Ax \neq 0$, then $|f(t)|$ is unbounded for $t \in \D^n$ near $x$. Let $(x_m) \subset \D^n$ be a sequence such that $(x_m) \to x$. Hence, for all $m$, $f(x_m) \in \D^n$ and in particular $||f(x_m)|| < 1$. This tells us that $||Ax_m|| < |\langle c,x_m \rangle + d|$ for all $m$; taking the limit as $m \to \infty$ forces $\langle c,x_m \rangle + d$ to approach $0$ and so $Ax = 0$. Invertibility of $A$ tells us that $x = 0$, which provides a contradiction since we assumed $x \in \partial \D^n$. Therefore, $f$ is well-defined upon extending to $\partial \D^n$ (and we call this extension $f$ without any harm) and must necessarily be continuous since it is a fractional linear map. Moreover, continuity forces $f(\partial \D^n) \subset \partial \D^n$ but since $f^{-1}$ is a map with the same properties as $f$, we actually have $f(\partial \D^n) = \partial \D^n$. Thus, from now on, we view $f$ as a map $\ol{\D^n} \to \ol{\D^n}$. 

Let $e_i \in \partial \D^n$ denote the $i$th standard basis vector for $i=1,\ldots,n$ and write $c = {}^t[c_1,\ldots,c_n] \in \R^n$. Using that $f(e_i), f(-e_i) \in \partial \D^n$, we have that
\begin{align*}
    \left|\left|\frac{A e_i}{c_i + d} \right|\right| = 1 \hspace{4em}
    \left|\left|\frac{A e_i}{c_i - d} \right|\right| = 1
\end{align*}
which tells us that $c_i + d = \pm (c_i - d)$. Since $d \neq 0$, we must have $c_i + d = - (c_i - d)$ i.e., $c_i = 0$. Since $i$ was arbitrary, we have that $f$ must be a map of the form $f(x) = \frac{Ax}{d}$ where $A:\R^n \to \R^n$ is invertible and linear, and $d \in \R$ is nonzero. Let $B = A/d : \R^n \to \R^n$ be a linear map. Then since $B (\partial \D^n) = f(\partial \D^n) = \partial \D^n$, $B$ must be a linear map preserving norms and hence preserves the inner product by the parallelogram identity. In other words, $B$ and hence our map $f$ is an orthogonal linear transformation. 
\end{proof}

\subsection{Proving isometry}\label{subsec:interval_preserving_to_isometry}\hrulefill

\begin{proof}[Proof of Theorem \ref{thm:main_omega}]
Let $f\colon \Omega \to \Omega$ be a bijection, mapping lines to lines as sets. By Theorem \ref{thm:interval_preserving}, for any $x \in \Omega$, we have
$$
f(x) = \frac{Ax + b}{ \langle c,x \rangle + d}
$$
for some fixed $A \in \Hom(\R^n,\R^n)$, $b,c \in \R^n$, and $d \in \R$. If ${}^t[x,1] \in \Omega \subset \R\P^n$ where ${}^t[x,1]$ is in projective coordinates and $\Omega$ is viewed in its ambient projective space, we have that 
$$
\begin{pmatrix}
A & b\\
{}^t c & d
\end{pmatrix}
\begin{bmatrix}
x \\
1
\end{bmatrix}
= \begin{bmatrix} A x + b \\ \langle c,x\rangle + d \end{bmatrix}
= \begin{bmatrix} \frac{A x + b}{\langle c,x\rangle + d} \\ 1 \end{bmatrix}
\in \Omega
$$
Viewing $T := \begin{pmatrix}A & b\\{}^t c & d\end{pmatrix}$ as a projective transformation on $\R\P^n$, $T(\Omega) = \Omega$ and so $T$ is an isometry of $(\Omega,d_\Omega)$. Let $\iota: \Omega \hookrightarrow \R\P^n$ be an embedding into the chart $\{x_{n+1} \neq 0\}$. Then we have the following commutative diagram:
\[
\tikzset{
  symbol/.style={
    draw=none,
    every to/.append style={
      edge node={node [sloped, allow upside down, auto=false]{$#1$}}}
  }
}
\begin{tikzcd}
& \Omega \arrow[swap]{d}{\iota} \arrow{r}{f} & \Omega \\
\R\P^n \arrow[r, symbol=\supset] & \Omega \arrow[swap]{r}{T} & \Omega \arrow[swap]{u}{\iota^{-1}} \arrow[r, symbol=\subset] & \R\P^n 
\end{tikzcd}
\]
Since $\iota,T$ are isometries with respect to the Hilbert metric, it follows that $f$ is an isometry. Explicitly, via $\iota$, $f$ can be extended to the projective transformation $T$ on $\R\P^n$.
\end{proof}

\section{Preliminaries on convex projective surfaces}\label{sec:prelims_surface}
\begin{definition}[Convex projective surface]\label{defn:convex_projective_surface}
Let $\Omega$ be a strictly convex set of $\R\P^2$ and $\Gamma$ a discrete subgroup of $\PGL(3,\R)$ that acts in a free and properly discontinuous manner on $\Omega$. Since the cross ratio is a projective invariant, $\Gamma$ acts by isometries, and so the Hilbert metric on $\Omega$ descends to a metric on $S := \Omega / \Gamma$. We call $S$ a \emph{convex projective surface} and the metric on $S$ is referred to as the \emph{Hilbert metric of $S$}. 
\end{definition}
\begin{definition}[Hyperbolic isometries]
Let $\Omega$ be a strictly convex set of $\R\P^2$. We call a transformation in $\PGL(3,\R)$ that fixes $\Omega$ as a set \emph{hyperbolic} if there are two distinct fixed points of the map on $\partial \Omega$.
\end{definition}

The covering map $\Omega \to S$ is a local isometry. An important note in Definition \ref{defn:convex_projective_surface} is that $S$ being closed of genus $g \geq 2$ tells us non-identity elements of $\Gamma$ are hyperbolic isometries. Furthermore, each free homotopy class of a nontrivial closed curve on our surface contains a unique closed geodesic (up to powers of this geodesic) \cite[Theorem 3.2(3)]{goldman_convex_compact}.

\noindent{\bf Note.} We shall only consider closed convex projective surfaces in this paper. 

Since $\Omega$ is simply connected, $\Gamma$ can be identified as $\pi_1(S)$ and each element $\gamma$ of $\Gamma$ corresponds to a complete geodesic in $\Omega$ on which $\gamma$ acts as translation as $\gamma$ fixes the endpoints of this line on $\partial \Omega$. The (complete) geodesics on $S$ are precisely the images of (complete) geodesics in $\Omega$ under the quotient map, and any lift of a (complete) geodesic on $S$ will provide a countable collection of (complete) geodesics on $\Omega$. Since the Klein-Beltrami model of hyperbolic space is exactly the Hilbert metric on the unit disk, our work to follow applies directly to hyperbolic surfaces, which are just quotients $\Omega / \Gamma$ for the particular case $\Omega = \D^2$. 

Our motivating question is the following (which parallels the question asked in the setting of convex sets as done in Section \ref{sec:prelims_convex}).

\question{Let $\Omega \subset \R\P^2$ be a strictly convex set endowed with the Hilbert metric and $\Gamma \subset \PGL(3,\R)$ be a discrete subgroup of isometries acting on $\Omega$ such that $S := \Omega / \Gamma$ is a closed convex projective surface. Suppose $f\colon S \to S$ is a bijection, mapping complete geodesics to complete geodesics as sets. Is $f$ an isometry for the Hilbert metric on $S$?}

\noindent{\bf Important.} Before we discuss our main result for such bijections on convex projective surfaces, we state a few points of clarification for the language used in this paper. First, the hypotheses on our map $f$ as mentioned in the question is a priori only concerned with geodesics as sets. With this in mind, we will treat geodesics that are the same as sets to be the same entity hence forth. We shall consider a geodesic as closed if it has a closed parametrization and we refer to a geodesic as simple if it has a simple parametrization. This makes the language to follow in the paper a little simpler.

To tackle our motivating question, we tried to understand the possible implications of this statement and perhaps properties of our map that might be simpler to recognize. A consequence of this statement would be our map is continuous and so, closed geodesics would map to closed geodesics. The remainder of Sections \ref{sec:prelims_surface} and \ref{sec:main} serves to dissect this latter statement and its validity; the main result is Theorem \ref{thm:main}, which characterizes equivalent ways of determining whether closed geodesics are preserved.

We now identify what would be a consequence if closed geodesics did not map to closed geodesics under our bijection. Let $S$ be a closed convex projective surface. Suppose $f\colon S \to S$ is a bijection, mapping complete geodesics to complete geodesics as sets, which induces a bijection on the set of geodesics of $S$. Let $\gamma$ be a closed geodesic. If $f(\gamma)$ is not closed, there exists some $y \in \ol{f(\gamma)} \setminus f(\gamma)$. By the smoothness of $f(\gamma)$ and of geodesics in general, we claim that every geodesic (except perhaps one) through $y$ intersects $f(\gamma)$ infinitely many times. Let $U$ be a local, convex chart containing $y$. Note that geodesic segments of $S$ contained in $U$ correspond to conventional straight lines based on our metric. Since $y$ is a limit point of $f(\gamma)$ and geodesics are smooth, any open, convex ball of $y$ contained in $U$ must necessarily contain a chord which is a segment of $f(\gamma)$. The chords can accumulate in a ``parallel" manner near $y$, which would determine a unique unoriented complete geodesic through $y$; this geodesic $\alpha$ (if it exists) is the unique one which may not intersect $f(\gamma)$. This is because unoriented geodesics that are not $\alpha$ and pass through $y$ must intersect infinitely many of the chords of $f(\gamma)$ since geodesic segments in $U$ are straight lines. Figure \ref{fig:closed_to_non_closed} illustrates a visual on what could be happening.

\begin{figure}[H]
\centering
\begin{tikzpicture}
\begin{scope}
\draw[smooth] (0,1) to[out=30,in=150] (2,1) to[out=-30,in=210] (3,1) to[out=30,in=150] (5,1) to[out=-30,in=30] (5,-1) to[out=210,in=-30] (3,-1) to[out=150,in=30] (2,-1) to[out=210,in=-30] (0,-1) to[out=150,in=-150] (0,1);
\draw[smooth] (0.4,0.1) .. controls (0.8,-0.25) and (1.2,-0.25) .. (1.6,0.1);
\draw[smooth] (0.5,0) .. controls (0.8,0.2) and (1.2,0.2) .. (1.5,0);
\draw[smooth] (3.4,0.1) .. controls (3.8,-0.25) and (4.2,-0.25) .. (4.6,0.1);
\draw[smooth] (3.5,0) .. controls (3.8,0.2) and (4.2,0.2) .. (4.5,0);
\node[red] at (0, 0.5) {$\gamma$};
\draw[red] (0, 0) arc(180:0:1 and 0.5);
\draw[red] (0, 0) arc(-180:0:1 and 0.5);
\end{scope}
\begin{scope}[xshift=5cm]
\draw[line width=0.8pt, ->] (1,0) -- (2,0);
\node at (1.5, 0.5) {$f$};
\end{scope}
\begin{scope}[xshift=8cm]
\draw[smooth] (0,1) to[out=30,in=150] (2,1) to[out=-30,in=210] (3,1) to[out=30,in=150] (5,1) to[out=-30,in=30] (5,-1) to[out=210,in=-30] (3,-1) to[out=150,in=30] (2,-1) to[out=210,in=-30] (0,-1) to[out=150,in=-150] (0,1);
\draw[smooth] (0.4,0.1) .. controls (0.8,-0.25) and (1.2,-0.25) .. (1.6,0.1);
\draw[smooth] (0.5,0) .. controls (0.8,0.2) and (1.2,0.2) .. (1.5,0);
\draw[smooth] (3.4,0.1) .. controls (3.8,-0.25) and (4.2,-0.25) .. (4.6,0.1);
\draw[smooth] (3.5,0) .. controls (3.8,0.2) and (4.2,0.2) .. (4.5,0);
\node[red] at (1, 0.75) {$f(\gamma)$};

\draw[hobby,red] plot coordinates {(5.5, 0.275) (5.25,0.4) (3,0.5) (1,0.5) (-0.25,0.4) (-0.5,0.275)};
\foreach \x in {0.125,0.0625,0.03125,0.015625} {
    \draw[red] (-0.5,\x) arc(180:0:0.51 and 0.2);
    \draw[red] (5.5,\x) arc(180:0:-0.51 and 0.2);
}
\node at (-0.125,-0.375) {$y$};
\node at (0,0.15)[circle,fill,inner sep=1.25pt] {};

\draw[dashed] (0,0.125) circle (0.3cm);
\path[draw, line width=0.5pt] (-0.25,0) -- (-1, -1.5);
\draw[dashed] (-1.5,-2.5) circle (1cm);
\draw[red,line width=0.2pt] (-2.5,-2.25) -- (-0.5,-2.25);
\draw[red,line width=0.2pt] (-2.5,-2.4) -- (-0.5,-2.4);
\draw[red,line width=0.2pt] (-2.55,-2.6) -- (-0.45,-2.6);
\draw[red,line width=0.2pt] (-2.55,-2.65) -- (-0.45,-2.65);
\draw[red,line width=0.3pt] (-2.55,-2.7) -- (-0.45,-2.7);
\draw[red,line width=0.4pt] (-2.55,-2.72) -- (-0.45,-2.72);
\node at (-1.5,-3) {$y$};
\node at (-1.5,-2.75)[circle,fill,inner sep=0.8pt] {};
\node at (-0.25, -1.75) {$U$};
\end{scope}
\end{tikzpicture}
\caption{If $f(\gamma)$ is not closed, then for any $y \in \ol{f(\gamma)} \setminus f(\gamma)$, every geodesic through $y$ except perhaps one intersects $f(\gamma)$ infinitely many times.}
\label{fig:closed_to_non_closed}
\end{figure}
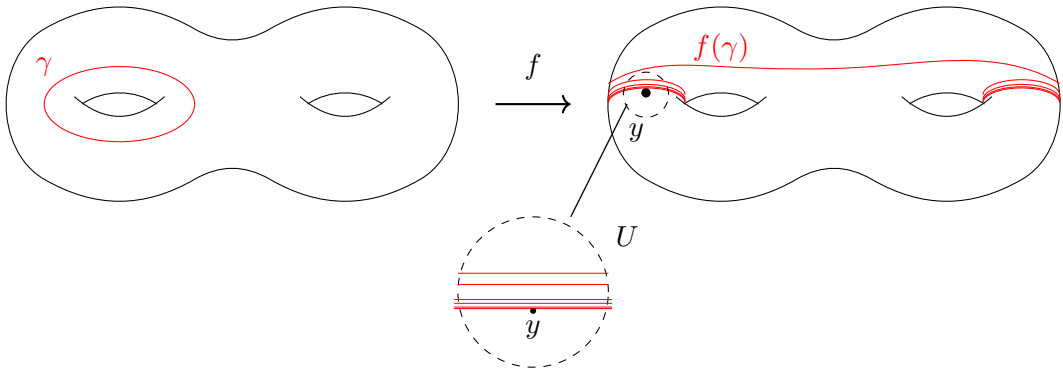
Thus, if for any point on our convex projective surface we could find at least two complete geodesics passing through the point and intersecting finitely many times, we would have that $f$ maps closed geodesics to closed geodesics. Note that this discussion has no assumptions added to our map $f$ itself. We are demanding more (if at all) from our convex projective surface and its distribution of geodesics to guarantee $f$ has the desired properties. It still remains to be answered whether any convex projective surface has this attribute.

We shall now discuss a few facts about (nontrivial) closed geodesics on our surface. Since $\Gamma$ is a discrete subgroup of $\PGL(3,\R)$, there are countably many closed geodesics on $S$. Moreover, \cite[Theorem 3.2(4)]{goldman_convex_compact} states that for any pair $(x,y) \in \partial\Omega\times\partial\Omega$, there exists a sequence in $\partial\Omega\times\partial\Omega$ consisting of the fixed points of a collection of isometries in $\Gamma$ that converges to $(x,y)$ (the precise formulation requires distinguishing the fixed points of elements of $\Gamma$ as attracting and repelling, which is defined in \cite[\S 1.9]{goldman_convex_compact}, but is unnecessary for our purposes). We claim that this proves the collection of points on closed geodesics forms a dense set of $S$. Indeed, consider any $x \in S$ and open neighborhood $U$ around $x$. We want to show that there exists a closed geodesic passing through $U$. Without loss of generality, suppose $U$ is isometric to a lifting $\widetilde{U}$ (shrink $U$ if necessary) and set a lifting of $x$ as $\tilde{x} \in \widetilde{U}$. Then just by considering any geodesic $\ell \subset \Omega$ through $\tilde{x}$, which of course passes through $\widetilde{U}$, and the corresponding endpoints of $\ell$ on $\partial\Omega$, \cite[Theorem 3.2(4)]{goldman_convex_compact} presents us with geodesics associated to elements of $\Gamma$ arbitrarily close to $\ell$ in the region $U$. Any of these geodesics descends to a closed geodesic intersecting $U$ on $S$ nontrivially.

In contrast, the collection of points that lie on simple closed geodesics form a nowhere dense set in $S$ and are of Hausdorff dimension $1$ in $S$ \cite[Theorem 1.11; Theorem 6.10]{birman_series_for_convex_projective}. 

Before we state our main result, we will require the notion of a filling collection (see \cite[\S 1.3.2]{primer} or \cite{kudlinska} for more), which is really independent of the convex projective structure on our surface as we shall see in Lemma \ref{lem:characterization_of_filling}.

\begin{definition}[Filling collections]\label{defn:filling}
Let $S$ be a closed convex projective surface. Let $\cC$ be a finite collection of closed, essential (i.e., homotopically nontrivial) geodesics on $S$. We say that $\cC$ is \emph{filling} or is a \emph{filling collection} if every non-trivial geodesic in $S$ intersects an element of $\cC$ transversely. We call a filling collection \emph{simple} if it consists of simple closed geodesics.
\end{definition}

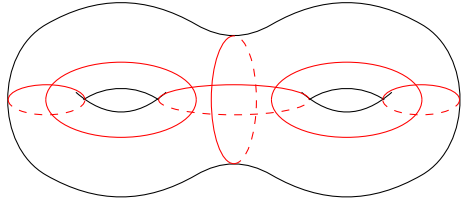
\begin{figure}[H]
\centering
\begin{tikzpicture}
\draw[smooth] (0,1) to[out=30,in=150] (2,1) to[out=-30,in=210] (3,1) to[out=30,in=150] (5,1) to[out=-30,in=30] (5,-1) to[out=210,in=-30] (3,-1) to[out=150,in=30] (2,-1) to[out=210,in=-30] (0,-1) to[out=150,in=-150] (0,1);
\draw[smooth] (0.4,0.1) .. controls (0.8,-0.25) and (1.2,-0.25) .. (1.6,0.1);
\draw[smooth] (0.5,0) .. controls (0.8,0.2) and (1.2,0.2) .. (1.5,0);
\draw[smooth] (3.4,0.1) .. controls (3.8,-0.25) and (4.2,-0.25) .. (4.6,0.1);
\draw[smooth] (3.5,0) .. controls (3.8,0.2) and (4.2,0.2) .. (4.5,0);
\draw[red] (-0.5,0) arc(180:0:0.51 and 0.2);
\draw[red,dashed] (-0.5,0) arc(180:0:0.51 and -0.2);
\draw[red] (2.5,-0.85) arc(270:90:0.3 and 0.85);
\draw[red,dashed] (2.5,-0.85) arc(270:450:0.3 and 0.85);
\draw[red] (5.5,0) arc(180:0:-0.51 and 0.2);
\draw[red,dashed] (5.5,0) arc(180:0:-0.51 and -0.2);
\draw[red] (1.5,0) arc(180:0:1 and 0.2);
\draw[red,dashed] (1.5,0) arc(180:0:1 and -0.2);
\draw[red] (0, 0) arc(180:0:1 and 0.5);
\draw[red] (0, 0) arc(-180:0:1 and 0.5);
\draw[red] (3, 0) arc(180:0:1 and 0.5);
\draw[red] (3, 0) arc(-180:0:1 and 0.5);
\end{tikzpicture}
\caption{Example of a simple filling collection on a genus 2 surface}
\label{fig:simple_filling_collection}
\end{figure}
Definition \ref{defn:filling} may not be how common literature states this concept, but we shall provide the necessary equivalences in Lemma \ref{lem:characterization_of_filling}; the lemma provides a neat characterization on determining whether a collection of geodesics is filling or not (\cite{kudlinska} provides an algorithm for this as well). As we shall see in Section \ref{sec:simple_implies_filling}, Figure \ref{fig:simple_filling_collection} represents a general fact: a closed surface always contains a simple filling collection.

\section{Main results for bijections on convex projective surfaces that preserve lines}\label{sec:main}

\begin{subsection}{Notation}
All closed geodesics will be essential (i.e., homotopically nontrivial). A subgroup of $\PGL(3,\R)$ consisting of hyperbolic isometries more carefully means that non-identity elements of the group are hyperbolic. Any pants decomposition considered will only be of simple closed geodesic curves. We may use the terms lines and geodesics interchangeably since geodesics with respect to the Hilbert metric on $\Omega$ are the usual restriction of lines in $\R^n$ and geodesics of $S$ are the projection of geodesics in $\Omega$. The following serves as important notation for this section:
\begin{table}[H]
\centering
\begin{tabular}{rcl}
$\Omega$ & $:$ & strictly convex set of $\R\P^2$\\
$\Gamma$ & $:$ & discrete subgroup of $\PGL(3,\R)$ that preserves $\Omega$ as a set, acts freely and properly\\
& & discontinuously on $\Omega$ and non-identity elements are hyperbolic isometries\\
$S$ & $:$ & a closed convex projective surface $\Omega / \Gamma$ of genus $g \geq 2$\\
$f$ & $:$ & bijection on $S$ mapping lines to lines that induces a bijection on the set of lines on $S$
\end{tabular}
\end{table}
\end{subsection}
Our main result for the surfaces setting is the following:
\begin{theorem}\label{thm:main}
Let $\Omega \subset \R\P^2$ be a strictly convex set endowed with the Hilbert metric and $\Gamma$ be a discrete subgroup of $\PGL(3,\R)$ fixing $\Omega$ as a set. Suppose that $\Omega / \Gamma$ is a closed surface with genus $g \geq 2$. If $f\colon S \to S$ is a bijection, mapping complete geodesics to complete geodesics as sets, which induces a bijection on the set of complete geodesics of $S$, then the following are equivalent:
\begin{enumerate}[(a)]
\item there exists a filling collection of closed geodesics mapped under $f$ to a collection of closed geodesics;
\item $f$ sends closed geodesics to closed geodesics;
\item $f$ sends simple closed geodesics to simple closed geodesics.
\end{enumerate}
\end{theorem}

We note the extra assumption of $f$ being a bijection on the set of geodesics; this guarantees $f^{-1}$ is a map with the same properties as $f$ (this was a redundant assumption for our study in $\Omega$ by Lemma \ref{lem:inverse_same_convex}). This hypothesis is added because there can exist several distinct geodesics between any two points.

As stated in Section \ref{sec:prelims_surface}, points on closed geodesics of $S$ form a dense set of $S$ and those on simple closed form a nowhere dense set. Hence, from a topological perspective, the equivalence of (b) and (c) seems notable. We prove this theorem by $(a) \Rightarrow (b)$ in Section \ref{sec:filling_implies_closed}, $(b) \Rightarrow (c)$ in Section \ref{sec:closed_implies_simple}, and $(c) \Rightarrow (a)$ in Section \ref{sec:simple_implies_filling}. Before we progress, we prove a quick corollary of this theorem.

\begin{definition}[Separating simple closed curves]
Let $S$ be a closed convex projective surface. Let $g \subset S$ be a simple closed curve. We say that $g$ is \emph{separating} if the complement of $g$ is the disjoint union of two subsurfaces.
\end{definition}

\begin{corollary}
Let $S$ be a closed convex projective surface. Suppose $f\colon S \to S$ is a bijection, mapping complete geodesics to complete geodesics as sets, which induces a bijection on the set of geodesics of $S$. If Theorem \ref{thm:main} holds, then $f$ sends separating simple closed geodesics to separating simple closed geodesics.
\end{corollary}
\begin{proof}
Let $\gamma$ be a separating simple closed geodesic. If $f(\gamma)$ is nonseparating, one can construct a simple closed curve $g$ intersecting $f(\gamma)$ once transversely (\cite[\S 1.3.2]{primer} -- stated for hyperbolic surfaces but a similar argument follows for the convex case) and so the associated geodesic to the free homotopy class of $g$ must intersect $f(\gamma)$ once transversely (\cite[Theorem 9.6.7]{ratcliffe} -- stated for hyperbolic surfaces but a similar argument follows for the convex case). Now, $f^{-1}(\alpha)$ is a simple closed geodesic that intersects $\gamma$ once. This cannot happen since $\gamma$ is separating simple so any other simple closed geodesic must intersect it an even number of times.
\end{proof}

A quick remark of this is that our map would hence send a collection of simple closed geodesics determining a pair of pants to another such collection (possibly the same).

\subsection{Preserving a filling collection implies preservation of closed geodesics}\hrulefill\label{sec:filling_implies_closed}

We first prove a lemma that characterizes filling collections. These collections are often defined to intersect every closed, essential curve (i.e., complement is a disjoint union of disks \cite{kudlinska}) and are only said to be simple; we show that this is equivalent to our Definition \ref{defn:filling} and we provide a stronger characterization of such collections in our particular setting, which crucially utilizes the structure of the geodesic flow on closed convex projective surfaces. 


\begin{lemma}\label{lem:characterization_of_filling}
Let $S$ be a closed convex projective surface. Let $\cF$ be a collection of closed geodesics on $S$. Then the following are equivalent:
\begin{enumerate}[(a)]
\item $\cF$ is filling;
\item every closed geodesic intersects an element of $\cF$ transversely;
\item every simple closed geodesic intersects an element of $\cF$ transversely;
\item every closed, essential curve intersects an element of $\cF$ transversely;
\item the complement $S \setminus \cF$ is a disjoint union of disks.
\end{enumerate}
\end{lemma}
\begin{proof}
It is clear that $(a)$ implies $(b)$ and $(b)$ implies $(c)$. Let $g \subset S$ denote a closed essential curve on $S$, and $\ell \subset g$ denote a simple closed curve. Then $\ell$ is freely homotopic to a unique simple closed geodesic $\gamma$. By assumption, $\gamma$ must intersect $\cF$ transversely -- say $\gamma$ intersects $\beta \in \cF$. The theory of geometric intersection number forces $\ell$ to intersect $\beta$ and hence $g$ to intersect $\beta$. This proves $(c) \Rightarrow (d)$. For $(d) \Rightarrow (e)$, if $S \setminus \cF$ is not a disjoint union of disks, take a connected component $K$ of this complement that is not a disk and consider a closed curve in $K$ generating a nontrivial subgroup of $\pi_1(K)$. Lastly, $(e) \Rightarrow (a)$ follows from applying Proposition \ref{prop:geods_in_disk} below.
\end{proof} 

\begin{proposition}\label{prop:geods_in_disk}
Let $S$ be a closed convex projective surface and $D \subset S$ denote a disk. Then $D$ contains no complete nontrivial geodesic and no non-closed geodesic can enter through $\partial D$ and stay in $D$ thereafter.
\end{proposition}
\begin{proof}
We prove the former statement (the latter is argued very similarly). It is clear that $D$ cannot contain a nontrivial closed geodesic since it is contractible. Let $\alpha$ be a non-closed geodesic contained in $D$. Then there exists an $x \in \ol{\alpha} \setminus \alpha$ and so $x \in \ol{D}$. Without loss of generality, $x \in \Int D$ by expanding $D$ slightly (Figure \ref{fig:non_closed_in_disk} provides a picture of what could be happening). For the remainder of this proof, let $d$ denote the Hilbert metric on $S$. Let $\varepsilon = \inf\{d(x,y) : y \in \partial D\}$, which must be positive since $\partial D \subset S$ is compact because $S$ is compact, and so $B_\varepsilon(x) \subset D$.
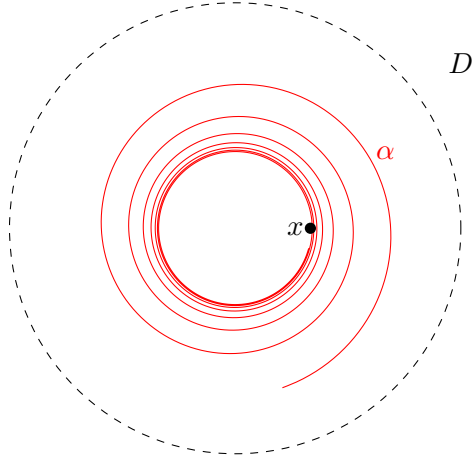
\begin{figure}[H]
\centering
\begin{tikzpicture}
\draw[dashed] (0,0) circle (3cm);
\node at (3,2.2) {$D$};
\draw [domain=5:50,variable=\t,smooth,samples=500,red]
        plot ({\t r}: {1+2*exp(-0.1*\t)});
\node[red] at (2,1) {$\alpha$};
\node at (0.8,0) {$x$};
\node at (1,0)[circle,fill,inner sep=1.5pt] {};
\end{tikzpicture}
\caption{Geodesic ray of a non-closed geodesic $\alpha \subset D$ with a limit point $x \in \Int D$}
\label{fig:non_closed_in_disk}
\end{figure}
We will now define the notion of a geodesic flow. Every unit (with respect to the norm associated to the Hilbert metric) tangent vector $v$ at a point $x \in \Omega$ completely characterizes a (complete) geodesic $r(u,v)$ in $\Omega$ passing through $x$ with $v$ parallel to it. For any $t \in \R$, we can consider the geodesic flow along $r(u,v)$ as follows: for every $t \in \R$, $(x_t,v_t)$ is the point in the unit tangent bundle of $\Omega$ given by $x_t \in r(x,v)$ at distance $t$ from $x$ and $v_t$ as the unit tangent vector such that $r(x_t,v_t) = r(x,v)$. The geodesic flow on our surface $S$ is simply the projection of the flow on $\Omega$. For a precise definition of geodesic flow on Hilbert geometries, see \cite{crampon2014geodesic}.

We will now need a powerful dynamical result, which guarantees the existence of a $\delta = \delta(\varepsilon) > 0$ with the following property: for any point $z \in S$ and number $\tau > \varepsilon$ with $d(z,\phi_\tau z) < \delta$, where $\phi_\tau$ denotes the geodesic flow on $S$ for time $\tau$, there exists a closed geodesic $\ell_0$ on $S$ such that the Hausdorff distance between the trajectory arc of time length $\tau$ with endpoints $z$ and $\phi_\tau z$ and the curve $\ell_0$ is less than $\varepsilon$. This result stated is one flavor of the \emph{Anosov closing lemma}, which can be applied since Benoist proved that the geodesic flow of the Hilbert metric on a closed convex projective surface is Anosov \cite{benoist}. See \cite[Lemma 4.1]{closing_lemma} for the general statement of this closing lemma.

Corresponding to $\varepsilon / 2$, we receive a $\delta > 0$ by the Anosov closing lemma. Let $\phi_t(y)$ be the geodesic flow along $\alpha$ for time $t$ from $y$. Since $x \in \ol{\alpha} \setminus \alpha$, we can pick $z \in \alpha$ and a $T > \varepsilon / 2$ such that (i) $d(z,\phi_T(z)) < \delta$; (ii) $z, \phi_T(z) \in B_{\varepsilon/2}(x)$; and (iii) $z, \phi_T(z)$ lie in some convex chart. Let $\ell$ denote the geodesic arc of $z$ and $\phi_T(z)$ along $\alpha$. By assumption, $\ell \subset B_{\varepsilon/2}(x) \subset D$ since $z,\phi_T(z)$ lie in a convex chart of our surface. Now, fix some $y \in \ell$ and let $w \in S$ such that $d(y,w) < \varepsilon / 2$. Then $d(x,w) \leq d(x,y) + d(y,w) < \varepsilon/2 + \varepsilon/2 = \varepsilon$. Hence, $w \in B_{\varepsilon}(x) \subset D$. Since $w$ was arbitrary, we have that for any $y \in \ell$, $B_{\varepsilon / 2}(y) \subset D$. By the Anosov closing lemma, we receive a periodic geodesic $\gamma$ such that $d_H(\gamma, \ell) < \varepsilon / 2$ where $d_H(\cdot,\cdot)$ denotes the Hausdorff distance. This tells us that $\gamma \subset \bigcup_{y \in \ell} B_{\varepsilon / 2}(y) \subset D$. This provides a contradiction since no closed, nontrivial geodesic can live in a disk.
\end{proof}

This next lemma proves a necessary and sufficient characterization of closed geodesics amongst arbitrary complete geodesics by filling collections.
\begin{lemma}\label{lem:characterization_of_closed_by_filling}
Let $S$ be a closed convex projective surface. Let $\cC$ be a filling collection on $S$ and $\alpha \subset S$ be an arbitrary geodesic such that $\alpha \not\in \cC$. Then $\alpha$ is closed if and only if $\alpha$ intersects elements of $\cC$ finitely many times.
\end{lemma}
\begin{proof}
Let $\alpha$ be a closed geodesic. The intersection of any two closed geodesics is finite and $\cC$ is a collection of finitely many geodesics; hence, the forward direction follows. Conversely, suppose $\alpha$ is non-closed and that $\alpha$ intersects $\cC$ finite many times. Proposition \ref{prop:geods_in_disk} tells us that $\alpha$ cannot live in some disk for all time after some point. Hence, $\alpha$ must contain at least two distinct geodesic paths between the same endpoints in some disk. This is not possible since every path in $S$ is homotopic relative to its endpoints to a unique geodesic path \cite[Proposition 3.1]{goldman_convex_compact}.
\end{proof}

\begin{proof}[Proof of Theorem \ref{thm:main}, $(a) \Rightarrow (b)$]
Let $\cC$ be a filling collection such that $f(\cC)$ consists of closed geodesics. Then necessarily $f(\cC)$ is filling since $f^{-1}$ is a map with the same properties as $f$. Let $\gamma$ be a closed geodesic. By Lemma \ref{lem:characterization_of_closed_by_filling}, $\gamma \cap \cC$ is finite. Since $f$ is a bijection, $f(\gamma) \cap f(\cC)$ is finite. Applying Lemma \ref{lem:characterization_of_closed_by_filling} again with respect to $f(\cC)$ tells us $f(\gamma)$ must be closed.
\end{proof}

\subsection{Preserving closed geodesics implies preservation of simple closed geodesics}\hrulefill
\label{sec:closed_implies_simple}

The main lemma needed to prove $(b)$ implies $(c)$ in Theorem \ref{thm:main} is the following:
\begin{lemma}\label{lem:intersect_simple_closed_but_not_simple}
Let $S$ be a closed convex projective surface. Let $\gamma,\alpha$ be distinct simple closed geodesics on $S$. Then there exists a closed geodesic $\beta \neq \gamma, \alpha$ and $\beta \cap \alpha \neq \emptyset$, $\beta \cap \gamma = \emptyset$.
\end{lemma}
\begin{proof}
If $\alpha$ and $\gamma$ are disjoint, then we can extend these curves to a pair of pants decomposition $\gamma_1 = \gamma, \gamma_2 = \alpha, \gamma_3,\ldots, \gamma_k$ of simple closed geodesics. Now, $\alpha$ serves as either the boundary of two pairs of pants or just one. Closed curves that intersect $\alpha$ and not $\gamma$ are constructed in Figure \ref{fig:alpha_gamma_disjoint}; necessarily the closed geodesic $\beta$ in the free homotopy class of these closed curves will have the same intersection relations with $\alpha$ and $\gamma$.
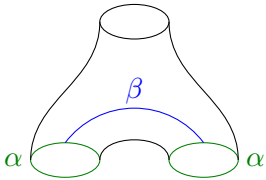
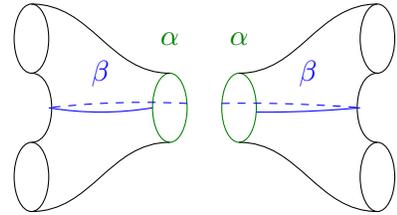
\begin{figure}[H]
\begin{subfigure}{0.4\textwidth}
\centering
    \begin{tikzpicture}[scale=0.92][xshift=-1cm]
    \draw (0,0) ellipse (.5 and .25);
    \draw[green!50!black] (-1,-2) ellipse (.5 and .25);
    \draw[green!50!black] (1,-2) ellipse (.5 and .25);
    \draw (-.5,0) to[out=-90,in=90] (-1.5,-2);
    \draw (.5,0) to[out=-90,in=90] (1.5,-2);
    \draw (-.5,-2) to[out=90,in=90] (.5,-2);
    \node[green!50!black] at (-1.75,-2) {$\alpha$};
    \node[green!50!black] at (1.75,-2) {$\alpha$};
    \draw[hobby,blue] plot coordinates {(-1,-1.75) (0,-1.25) (1,-1.75)};
    \node[blue] at (0,-1) {$\beta$};
    \end{tikzpicture}
\caption{Case 1: $\alpha$ is the boundary of one pant}
\end{subfigure}%
\hfill
\begin{subfigure}{0.4\textwidth}
\centering
\begin{tikzpicture}[scale=0.92]
    \begin{scope}[rotate=-90]
    \draw[green!50!black] (0,0) ellipse (.5 and .25);
    \draw (-1,-2) ellipse (.5 and .25);
    \draw (1,-2) ellipse (.5 and .25);
    \draw (-.5,0) to[out=-90,in=90] (-1.5,-2);
    \draw (.5,0) to[out=-90,in=90] (1.5,-2);
    \draw (-.5,-2) to[out=90,in=90] (.5,-2);
    \node[green!50!black] at (-1,0) {$\alpha$};
    \node[blue] at (-0.5,-1) {$\beta$};
    \draw[hobby,blue,dashed] plot coordinates {(-0.0625,0.25) (-0.0625,-1) (0,-1.75)};
    \draw[hobby,blue] plot coordinates {(0,-0.25) (0.0625,-1) (0, -1.75)};
    \end{scope}
    \begin{scope}[xshift=1cm,rotate=90]
    \draw[green!50!black] (0,0) ellipse (.5 and .25);
    \draw (-1,-2) ellipse (.5 and .25);
    \draw (1,-2) ellipse (.5 and .25);
    \draw (-.5,0) to[out=-90,in=90] (-1.5,-2);
    \draw (.5,0) to[out=-90,in=90] (1.5,-2);
    \draw (-.5,-2) to[out=90,in=90] (.5,-2);
    \node[green!50!black] at (1,0) {$\alpha$};
    \node[blue] at (0.5,-1) {$\beta$};
    \draw[hobby,blue,dashed] plot coordinates {(0.0625,0.25) (0.0625,-0.5) (0,-1.75)};
    \draw[hobby,blue] plot coordinates {(-0.0625,-0.25) (-0.0625,-0.5) (0, -1.75)};
    \end{scope}
    \end{tikzpicture}
\caption{Case 2: $\alpha$ is the boundary of two pants}
\end{subfigure}%
\caption{Construction of a closed geodesic $\beta$ such that $\beta$ intersects $\alpha$ but not $\gamma$ if $\alpha \cap \gamma = \emptyset$}
\label{fig:alpha_gamma_disjoint}
\end{figure}
If $\alpha$ and $\gamma$ are not disjoint, then extend $\gamma$ to a pair of pants $\gamma_1 = \gamma, \gamma_2,\ldots,\gamma_k$. If $\alpha$ intersects some $\gamma_i$ for $i \neq 1$, then just take $\beta = \gamma_i$. Otherwise, $\alpha \cap \gamma_i = \emptyset$ for all $i \neq 1$. We treat the individual cases of whether $\gamma$ bounds 1 or 2 pants in the decomposition. Construction of the closed geodesic $\beta$ is denoted in Figure \ref{fig:alpha_gamma_intersect} for each scenario.
\begin{figure}[H]
\begin{subfigure}{0.5\textwidth}
\centering
    \begin{tikzpicture}[scale=0.92]
    \draw (0,0) ellipse (.5 and .25);
    \draw[green!50!black] (-1,-2) ellipse (.5 and .25);
    \draw[green!50!black] (1,-2) ellipse (.5 and .25);
    \draw (-.5,0) to[out=-90,in=90] (-1.5,-2);
    \draw (.5,0) to[out=-90,in=90] (1.5,-2);
    \draw (-.5,-2) to[out=90,in=90] (.5,-2);
    \node[green!50!black] at (-1.75,-2) {$\gamma$};
    \node[green!50!black] at (1.75,-2) {$\gamma$};
    \draw[hobby,blue,dashed] plot coordinates {(-0.0625,0.25) (-0.0625,-1) (0,-1.75)};
    \draw[hobby,blue] plot coordinates {(0,-0.25) (0.0625,-1) (0, -1.75)};
    \node[blue] at (-0.5,-1) {$\beta$};
    \end{tikzpicture}
\caption{Case 1: $\gamma$ is the boundary of one pant}
\end{subfigure}%
\hfill
\begin{subfigure}{0.5\textwidth}
\centering
    \begin{tikzpicture}[scale=0.92]
    \draw[green!50!black] (0,0) ellipse (.5 and .25);
    \draw[red] (-1,-2) ellipse (.5 and .25);
    \draw[red] (1,-2) ellipse (.5 and .25);
    \draw (-.5,0) to[out=-90,in=90] (-1.5,-2);
    \draw (.5,0) to[out=-90,in=90] (1.5,-2);
    \draw (-.5,-2) to[out=90,in=90] (.5,-2);
    \node[green!50!black] at (-1,0) {$\gamma$};
    \draw[hobby,blue] plot coordinates {(-1,-1.75) (0,-1.25) (1,-1.75)};
    \node[blue] at (0,-1) {$\beta$};
    \end{tikzpicture}
\caption{Case 2: $\gamma$ is the boundary of two pants and other two boundary components are identified}
\end{subfigure}%
\hfill\vspace{2em}
\begin{subfigure}{0.8\textwidth}
\centering
    \begin{tikzpicture}[scale=0.92]
    \begin{scope}[rotate=90]
    \draw[green!50!black] (0,0) ellipse (.5 and .25);
    \draw[red] (-1,-2) ellipse (.5 and .25);
    \draw[red,dashed] (1,-2) ellipse (.5 and .25);
    \draw (-.5,0) to[out=-90,in=90] (-1.5,-2);
    \draw (.5,0) to[out=-90,in=90] (1.5,-2);
    \draw (-.5,-2) to[out=90,in=90] (.5,-2);
    \node[green!50!black] at (1,0) {$\gamma$};
    \draw[hobby,blue] plot coordinates {(-1,-1.75) (0,-1.25) (1,-1.75)};
    \node[blue] at (0,-1) {$\beta$};
    \end{scope}
    \begin{scope}[xshift=5cm,rotate=-90]
    \draw (0,0) ellipse (.5 and .25);
    \draw[red,dashed] (-1,-2) ellipse (.5 and .25);
    \draw[red] (1,-2) ellipse (.5 and .25);
    \draw (-.5,0) to[out=-90,in=90] (-1.5,-2);
    \draw (.5,0) to[out=-90,in=90] (1.5,-2);
    \draw (-.5,-2) to[out=90,in=90] (.5,-2);
    \draw[hobby,blue] plot coordinates {(-1,-1.75) (0,-1.25) (1,-1.75)};
    \node[blue] at (0,-1) {$\beta$};
    \end{scope}
    \end{tikzpicture}
\caption{Case 3: $\gamma$ is the boundary of two pants and other two boundary components are not identified}
\end{subfigure}%
\caption{Construction of a closed geodesic $\beta$ such that $\beta$ intersects $\alpha$ but not $\gamma$ if $\alpha \cap \gamma \neq \emptyset$, $\alpha \cap \gamma_i =\emptyset$ for $i \neq 1$}
\label{fig:alpha_gamma_intersect}
\end{figure}
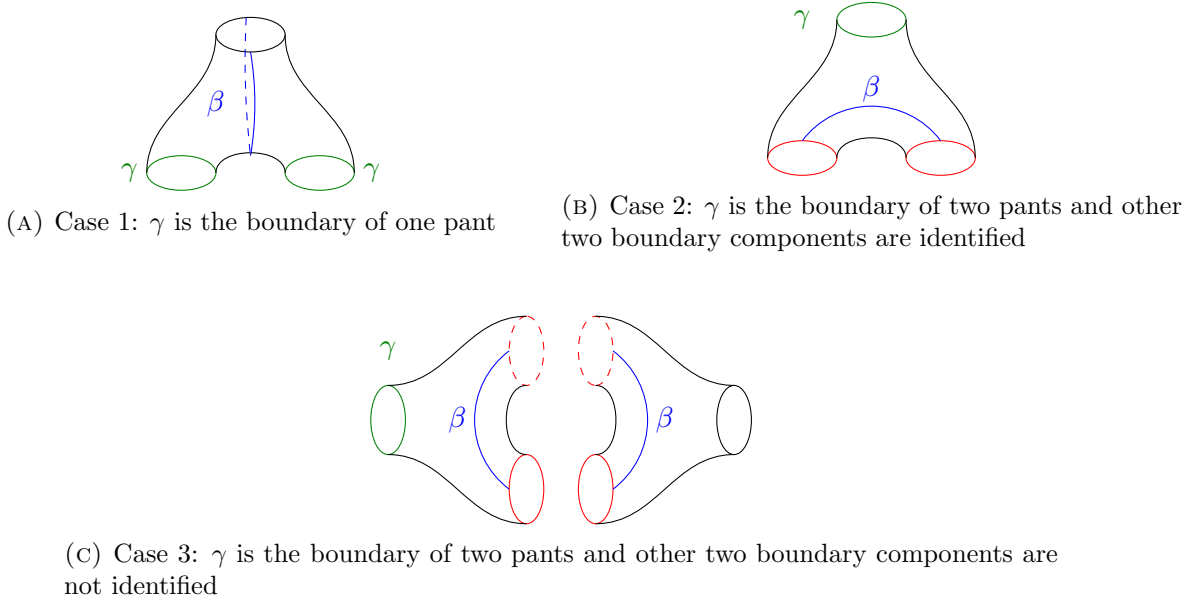
Now, $\beta$ must necessarily intersect $\alpha$ because otherwise, $\alpha$ would have to be freely homotopic to (a power of) one of the boundary components of our pant and by assumption, this cannot happen.
\end{proof}

\begin{corollary}\label{cor:intersect_closed_but_not_simple}
Let $S$ be a closed convex projective surface. Let $\gamma,\alpha$ be distinct closed geodesics on $S$ where $\gamma$ is simple. Then there exists a closed geodesic $\beta \neq \gamma, \alpha$ and $\beta \cap \alpha \neq \emptyset$, $\beta \cap \gamma = \emptyset$.
\end{corollary}
\begin{proof}
Since $\alpha$ is a closed geodesic, there is some simple closed curve $g \subset \alpha$ and $g$ is freely homotopic to a simple closed geodesic $\alpha'$ that is distinct as a set to $\gamma$ (if every simple closed curve of $\alpha$ is freely homotopic to $\gamma$ then $\alpha$ and $\gamma$ are the same as sets). Any closed geodesic intersecting $\alpha'$ must intersect $g$ and hence $\alpha$ by the theory of geometric intersection numbers. Thus, it suffices to show that there exists a closed geodesic intersecting $\alpha'$ but not $\gamma$. This is the content of Lemma \ref{lem:intersect_simple_closed_but_not_simple}.
\end{proof}

\begin{proof}[Proof of Theorem \ref{thm:main}, $(b) \Rightarrow (c)$]
Let $\gamma$ be a simple closed geodesic. If $f(\gamma)$ is not simple, then there exists a simple closed curve $g \subset f(\gamma)$ and $g$ is freely homotopic to a unique simple closed geodesic $\gamma'$. 
\begin{figure}[H]
\centering
\begin{tikzpicture}[scale=0.95]
\begin{scope}
\draw[smooth] (0,1) to[out=30,in=150] (2,1) to[out=-30,in=210] (3,1) to[out=30,in=150] (5,1) to[out=-30,in=30] (5,-1) to[out=210,in=-30] (3,-1) to[out=150,in=30] (2,-1) to[out=210,in=-30] (0,-1) to[out=150,in=-150] (0,1);
\draw[smooth] (0.4,0.1) .. controls (0.8,-0.25) and (1.2,-0.25) .. (1.6,0.1);
\draw[smooth] (0.5,0) .. controls (0.8,0.2) and (1.2,0.2) .. (1.5,0);
\draw[smooth] (3.4,0.1) .. controls (3.8,-0.25) and (4.2,-0.25) .. (4.6,0.1);
\draw[smooth] (3.5,0) .. controls (3.8,0.2) and (4.2,0.2) .. (4.5,0);
\node[red] at (0, 0.5) {$\gamma$};
\draw[red] (0, 0) arc(180:0:1 and 0.5);
\draw[red] (0, 0) arc(-180:0:1 and 0.5);
\end{scope}
\begin{scope}[xshift=5cm]
\draw[line width=0.8pt, ->] (1,0) -- (2,0);
\node at (1.5, 0.5) {$f$};
\end{scope}
\begin{scope}[xshift=8cm]
\draw[smooth] (0,1) to[out=30,in=150] (2,1) to[out=-30,in=210] (3,1) to[out=30,in=150] (5,1) to[out=-30,in=30] (5,-1) to[out=210,in=-30] (3,-1) to[out=150,in=30] (2,-1) to[out=210,in=-30] (0,-1) to[out=150,in=-150] (0,1);
\draw[smooth] (0.4,0.1) .. controls (0.8,-0.25) and (1.2,-0.25) .. (1.6,0.1);
\draw[smooth] (0.5,0) .. controls (0.8,0.2) and (1.2,0.2) .. (1.5,0);
\draw[smooth] (3.4,0.1) .. controls (3.8,-0.25) and (4.2,-0.25) .. (4.6,0.1);
\draw[smooth] (3.5,0) .. controls (3.8,0.2) and (4.2,0.2) .. (4.5,0);

\draw[blue] (3, 0) arc(180:0:1 and 0.5);
\draw[blue] (3, 0) arc(-180:0:1 and 0.5);
\node[blue] at (4, -0.75) {$\gamma'$};

\coordinate (A) at (2.5,0);
\coordinate (A') at (3,0.6);
\coordinate (A'') at (3,-0.6);
\coordinate (B) at (3.75,1);
\coordinate (C) at (5.2,0);
\coordinate (D) at (3.75,-1);

\draw[closed hobby,red] plot coordinates {(-0.2,0) (1.25,1) (2.5,0) (3.75,-1) (5.2,0) (3.75,1) (1.25,-1)};
\node[red] at (1.25,0.5) {$f(\gamma)$};

\draw[green!50!black,thick,dashed] (A) -- (A) to [curve through={ (A') (B) (C) (D) (A'')}] (A);
\node[green!50!black,thick] at (5.5,1) {$g$};
\end{scope}
\end{tikzpicture}
\caption{If $f(\gamma)$ is not simple, there exists a simple closed geodesic $\gamma'$ such that any closed geodesic intersecting $\gamma'$ intersects $f(\gamma)$ transversely.}
\label{fig:closed_to_simple}
\end{figure}
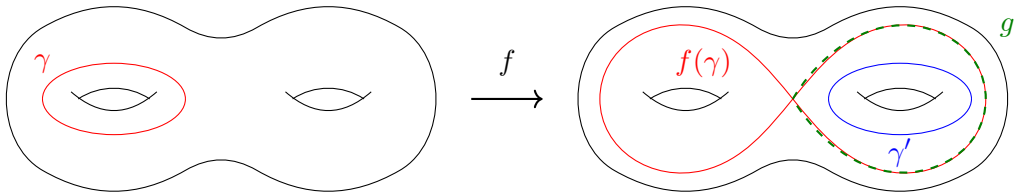
The key fact is that any closed geodesic intersecting $\gamma'$ must intersect $g$ and hence $f(\gamma)$ because of geometric intersection numbers being realized by geodesics. By assumption, $f^{-1}(\gamma')$ is a closed geodesic not equal to $\gamma$ and any closed geodesic intersecting $f^{-1}(\gamma')$ intersects $\gamma$ transversely. Corollary \ref{cor:intersect_closed_but_not_simple} tells us that this cannot happen.
\end{proof}

\subsection{Preserving simple closed geodesics implies preservation of some filling collection}\hrulefill
\label{sec:simple_implies_filling}

\begin{proof}[Proof of Theorem \ref{thm:main}, $(c) \Rightarrow (a)$]
It suffices to show that there exists a simple filling collection (recall Definition \ref{defn:filling}). Let $\gamma_1,\ldots,\gamma_k$ denote a disjoint collection of simple closed geodesics providing a pants decomposition of the surface $S$. 
\begin{figure}[H]
\centering
\begin{tikzpicture}[scale=0.92]
\begin{scope}[rotate=-90]
\draw (0,0) ellipse (.5 and .25);
\draw (-1,-2) ellipse (.5 and .25);
\draw (1,-2) ellipse (.5 and .25);
\draw (-.5,0) to[out=-90,in=90] (-1.5,-2);
\draw (.5,0) to[out=-90,in=90] (1.5,-2);
\draw (-.5,-2) to[out=90,in=90] (.5,-2);
\node at (-1,0) {$\gamma_i$};
\node[blue] at (-0.5,-1) {$\beta$};
\draw[hobby,blue,dashed] plot coordinates {(-0.0625,0.25) (-0.0625,-1) (0,-1.75)};
\draw[hobby,blue] plot coordinates {(0,-0.25) (0.0625,-1) (0, -1.75)};
\end{scope}
\begin{scope}[xshift=1cm,rotate=90]
\draw (0,0) ellipse (.5 and .25);
\draw (-1,-2) ellipse (.5 and .25);
\draw (1,-2) ellipse (.5 and .25);
\draw (-.5,0) to[out=-90,in=90] (-1.5,-2);
\draw (.5,0) to[out=-90,in=90] (1.5,-2);
\draw (-.5,-2) to[out=90,in=90] (.5,-2);
\node at (1,0) {$\gamma_i$};
\node[blue] at (0.5,-1) {$\beta$};
\draw[hobby,blue,dashed] plot coordinates {(0.0625,0.25) (0.0625,-0.5) (0,-1.75)};
\draw[hobby,blue] plot coordinates {(-0.0625,-0.25) (-0.0625,-0.5) (0, -1.75)};
\end{scope}
\end{tikzpicture}
\caption{Extending a pair of pants collection of geodesics to a filling collection}
\label{fig:extending_pants_to_filling}
\end{figure}
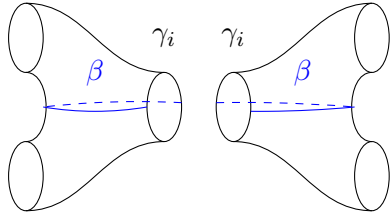
For each pant, consider the curve $\beta$ as in Figure \ref{fig:extending_pants_to_filling}. This curve will be freely homotopic to some simple closed geodesic. Considering all these geodesics for all the pants along with all the $\gamma_j$, we claim that this collection will be filling. This follows because any simple closed geodesic in the complement of our collection on the surface must be freely homotopic to one of our closed curves in the collection. Now apply Lemma \ref{lem:characterization_of_filling}.
\end{proof}

\section{Geodesic laminations}\label{sec:geodesic_laminations}

\begin{definition}[Geodesic laminations]
Let $S$ be a closed convex projective surface. A \emph{geodesic lamination} $\lambda$ of $S$ is a set of disjoint simple (not necessarily closed) complete geodesics in $S$ whose union is a closed set of $S$. Each geodesic in $\lambda$ is called a \emph{leaf}. A \emph{minimal lamination} is a geodesic lamination containing no proper non-empty sublamination.
\end{definition}
As an example, any finite collection of disjoint, simple closed geodesics on $S$ defines a lamination. Another example is provided in Figure \ref{fig:geod_lamination}. Basic results of such objects and their role in understanding surfaces is discussed in \cite[\S 8.3, 8.4]{martelli}.

\begin{figure}[H]
\centering
\begin{tikzpicture}[scale=0.95]
\draw[smooth] (0,1) to[out=30,in=150] (2,1) to[out=-30,in=210] (3,1) to[out=30,in=150] (5,1) to[out=-30,in=30] (5,-1) to[out=210,in=-30] (3,-1) to[out=150,in=30] (2,-1) to[out=210,in=-30] (0,-1) to[out=150,in=-150] (0,1);
\draw[smooth] (0.4,0.1) .. controls (0.8,-0.25) and (1.2,-0.25) .. (1.6,0.1);
\draw[smooth] (0.5,0) .. controls (0.8,0.2) and (1.2,0.2) .. (1.5,0);
\draw[smooth] (3.4,0.1) .. controls (3.8,-0.25) and (4.2,-0.25) .. (4.6,0.1);
\draw[smooth] (3.5,0) .. controls (3.8,0.2) and (4.2,0.2) .. (4.5,0);

\draw[hobby,red] plot coordinates {(5.5, 0.275) (5.25,0.4) (3,0.5) (1,0.5) (-0.25,0.4) (-0.5,0.275)};
\foreach \x in {0.125,0.0625,0.03125,0.015625} {
    \draw[red] (-0.5,\x) arc(180:0:0.51 and 0.2);
    \draw[red] (5.5,\x) arc(180:0:-0.51 and 0.2);
}
\draw[blue,line width=1pt] (-0.5,-0.025) arc(180:0:0.51 and 0.2);
\draw[green!50!black,line width=1pt] (5.5,-0.025) arc(180:0:-0.51 and 0.2);
\end{tikzpicture}
\caption{Example of a geodesic lamination with non-closed geodesics}
\label{fig:geod_lamination}
\end{figure}
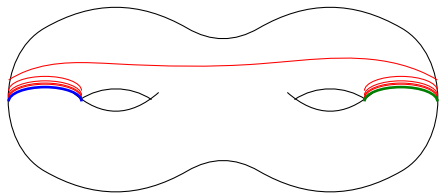

Since the closure of any leaf $\ell \subset \lambda$ is a sub-lamination of $\lambda$, any $\lambda$ is a minimal lamination if and only if each leaf in $\lambda$ is dense in $\lambda$. Precisely, a minimal lamination can only be a simple closed geodesic or a collection of uncountably many non-closed geodesics, each of which is dense in the lamination \cite[Lemma I.4.2.2]{canary2006fundamentals}.

\begin{lemma}\label{lem:minimal_lams_to_minimal_lams}
Let $S$ be a closed convex projective surface. Let $h\colon S \to S$ be a bijection, mapping geodesic laminations to geodesic laminations that induces a bijection on the set of laminations of $S$. Then $h$ sends minimal laminations to minimal laminations.
\end{lemma}
\begin{proof}
Let $\lambda$ be a minimal lamination. Suppose $h(\lambda)$ is not a minimal lamination. Then there exists a proper sub-lamination $\lambda' \subsetneq h(\lambda)$. Since $h^{-1}$ is a map with the same properties as $h$, we would have a contradiction to the minimality of $\lambda$ upon applying $h^{-1}$.
\end{proof}

\begin{lemma}\label{lem:laminations_implies_simple}
Let $S$ be a closed convex projective surface. Suppose $h\colon S \to S$ be a bijection, mapping geodesic laminations to geodesic laminations that induces a bijection on the set of laminations of $S$. Then $h$ sends simple closed geodesics to simple closed geodesics.
\end{lemma}
\begin{proof}
Let $\gamma$ be a simple closed geodesic. Since $\gamma$ is a minimal lamination, $h(\gamma)$ must be a simple closed geodesic or an uncountable collection of non-closed geodesics, each of which is dense in the lamination by Lemma \ref{lem:minimal_lams_to_minimal_lams} and above remarks. Suppose the latter held. Write $h(\gamma) = \bigsqcup_{i \in I} \alpha_i$ as our lamination where $I$ is an uncountable indexing set and each $\alpha_i$ is a non-closed geodesic. Let $\beta$ be a simple closed geodesic on $S$ such that $\beta$ intersects $\gamma$ once or twice (depending on if $\gamma$ is separating), and so $h(\beta)$ is a minimal lamination as well. Let $x \in \gamma \cap \beta$. Say $h(x)$ belongs to some $\alpha_i \subset h(\gamma)$. Then by hypothesis, $h(x)$ can be approximated arbitrarily closely by $\alpha_j$ for any $j \neq i$ (and note $\alpha_j \cap \alpha_i = \emptyset$). Therefore, any geodesic except $\alpha_i$ passing through $h(x)$ intersects the collection of $\alpha_j$ for $j \neq i$ infinitely many times. But $h(\beta)$ contains a geodesic distinct from $\alpha_i$ that passes through $h(x)$ and intersects $h(\gamma)$ finitely many times, which provides a contradiction.
\end{proof}

The main statement concerning geodesic laminations is Proposition \ref{prop:lamination_implies_main} below. We note that Lemma \ref{lem:laminations_implies_simple} with the assumptions of the below proposition already guarantee Theorem \ref{thm:main} to apply; we now provide a quicker, more direct proof of this fact.

\begin{proposition}\label{prop:lamination_implies_main}
Let $S$ be a closed convex projective surface. Suppose $f\colon S \to S$ is a bijection, mapping complete geodesics to complete geodesics as sets, which induces a bijection on the set of these geodesics. Suppose in addition that $f$ sends geodesic laminations to geodesic laminations and induces a bijection on the set of laminations of $S$. Then $f$ maps closed geodesics to closed geodesics and simple closed geodesics to simple closed geodesics (i.e., Theorem \ref{thm:main} holds for $f$).
\end{proposition}
\begin{proof}
Let $\gamma$ be a simple closed geodesic. Then $f(\gamma)$ is a minimal lamination; thus, $f(\gamma)$ must be a simple closed geodesic or a disjoint collection of geodesics. The assumption that $f$ maps lines to lines tells us the former must hold. 
\end{proof}

\section{Future directions}\label{sec:future_directions}
Let $S$ be a closed convex projective surface. Suppose $f\colon S \to S$ is a bijection, mapping complete geodesics to complete geodesics as sets, which induces a bijection on the set of geodesics of $S$. A satisfying picture to have would be that our map $f$ is induced by a bijection $\Omega \to \Omega$ that also sends lines to lines. If $f$ is continuous, then we argue now that this is necessarily true: $f$ must be a homeomorphism as $f^{-1}$ is a map with the same properties as $f$ so $f^{-1}$ is continuous as well. Hence, the lifting criterion tells us that we can lift $f$ to a homeomorphism $\tilde{f}: \Omega \to \Omega$ where $\tilde{f}^{-1} = \widetilde{f^{-1}}$. Let $\pi: \Omega \to S$ denote the quotient map. To see that $\tilde{f}$ maps lines to lines, we note that $\pi \circ \tilde{f} = f \circ \pi$ so $\pi = f \circ \pi \circ \widetilde{f^{-1}}$. Since $\pi(\ell)$ is a line in $S$ iff $\ell$ is a line in $\Omega$ and the same holds for $f$ (i.e., for $\ell' \subset S$, $f(\ell')$ is a line iff $\ell'$ is a line), it follows that $\widetilde{f^{-1}}$ preserves lines and hence so does $\tilde{f}$.

What would of course be more interesting is whether the map $f\colon S \to S$ is always continuous; this still remains to be answered. We outline and include a few questions (some of which may have been mentioned above) that might help understand properties of our bijection (if at all). 
\begin{itemize}
\item Does Theorem \ref{thm:main} always hold i.e., does $f$ always map closed geodesics to closed geodesics?
\item Is $f$ continuous on geodesics? Is $f$ continuous on all of $S$?
\item Is $f$ an isometry of the Hilbert metric on $S$?
\item Is $f$ a descent of an isometry $\Omega \to \Omega$ where $\Omega \subset \R\P^2$ is a strictly convex set as the universal cover of $S$?
\end{itemize}
If the answer to the last question is in the affirmative, then this result would parallel the fundamental theorem of projective geometry.

\bibliographystyle{alpha}
\bibliography{ref}

\end{document}